\documentclass[11pt]{article}
\pdfoutput=1
\usepackage{amsmath,amssymb,amsthm}
\usepackage[draft=false,pdftex=true,unicode=true,linkcolor=blue,colorlinks=true, hyperfootnotes=false]{hyperref}
\usepackage{graphicx}
\usepackage{indentfirst}

\usepackage[usenames]{color}

\newcommand{\Max}{\max\limits}
\DeclareMathOperator*{\AArgmax}{Argmax}
\newcommand{\Argmax}{\AArgmax\limits}

\DeclareMathOperator*{\diam}{diam}
\renewcommand{\le}{\leqslant}
\renewcommand{\ge}{\geqslant}
\newtheorem{theorem}{Theorem}

\newtheorem{stat}{Proposition}
\newtheorem{lemma}{Lemma}
\newenvironment{system}[1]{\left\{ \begin{array} {#1}}{\end{array} \right.}
\newcommand{\NN}{\mathbb{N}}
\newcommand{\txt}[1]{{\scriptsize\mbox{\textnormal{#1}}}}

\begin{document}
\title{Sharp bounds for the number of maximal independent sets in trees of fixed diameter}

\author{
Alexander B. Dainiak\thanks{
Supported by RFBR grant 07-01-00444.}\\
Department of Computational Mathematics and Cybernetics \\
Lomonosov Moscow State University \\
}

\maketitle
\thispagestyle{empty}

\begin{abstract}
We obtain sharp lower and upper bounds for the number of maximal (under inclusion) independent sets in trees with fixed number of vertices and diameter.
All extremal trees are described up to isomorphism.
\end{abstract}

\subsection*{Introduction}
A subset of pairwise disjoint vertices of a graph is called an \emph{independent set}. We call and independent set a \emph{maximal independent set} (m.\,i.\,s.) if it is
not contained in an independent set of greater size.
Let $i_M(G)$ denote the number of m.i.s. in $G$, and $n(G)$ denote the number of vertices of $G$. The sets of vertices and edges of a graph $G$ are denoted by $V(G)$ and $E(G)$ respectively.
Let $\partial v$ denote the set of all neighbors of $v$. Write $G\setminus S$ for the subgraph of $G$, generated by a set $V(G)\setminus S$.

The number of edges in a path is called its \emph{length}. A \emph{diameter} of a tree $T$ is the maximal length $\diam(T)$ of a path in $T$.
Let $d,\,n\in\NN$, and let $d<n$. Every tree of diameter $d$ on $n$ vertices, having the minimal (maximal) number of m.\,i.\,s. among all trees of the same diameter and number of vertices,
is called the \emph{$(n,d)_{\txt{m.i.s.}}$--minimal} (respectively, \emph{$(n,d)_{\txt{m.i.s.}}$--maximal}).

The problem of counting independent sets in trees of fixed diameter was addressed in \cite{Pedp}, where sharp upper bound was provided for the number of all (not necessarily maximal)
independent sets in trees with a given number of vertices. The extremal trees were characterized up to isomorphism. In the same paper the problem of finding lower bounds for the number
of independent sets in trees of fixed diameter was raised. Although the latter problem is still open, some progress was achieved in \cite{Frendp} and \cite{DainTp}.
\par The purpose of the current report is to solve the similar problem for maximal independent sets. The lower bound is almost trivial, while the upper bound requires some effort.
Sharp upper bounds for the number of m.\,i.\,s. in trees without any restrictions on diameter was obtained by H.\,S.\,Wilf in \cite{Wilfpaper}, and B.\,Sagan in \cite{Saganpaper2}
described the structure of extremal trees. Below we provide the complete characterization of $(n,d)_{\txt{m.i.s.}}$--maximal trees, thus generalizing the results of Wilf and Sagan.

By $\psi_n$ we denote the number of vertices in a path on $n$ vertices. The sequence $\psi_n$, obviously, can be defined by the recurrence $\psi_n=\psi_{n-2}+\psi_{n-3}$ and the initial
conditions $\psi_0=\psi_1=1,\,\psi_2=2$. The following table shows the values of $\psi_n$ for small $n$:
\begin{center}
\begin{tabular}{|c|c|c|c|c|c|c|c|c|c|c|c|c|c|c|c|c|c|c|c|c|c|}\hline
$n$&$0$&$1$&$2$&$3$&$4$&$5$&$6$&$7$&$8$&$9$&$10$&$11$&$12$&$13$&$14$&$15$\\ \hline
$\psi_n$&$1$&$1$&$2$&$2$&$3$&$4$&$5$&$7$&$9$&$12$&$16$&$21$&$28$&$37$&$49$&$65$\\ \hline
\end{tabular}
\end{center}

\subsection*{Bounds for the number of maximal independent sets}
Let $U=\{u_1,\,\ldots,\,u_{d-1}\}$, $V=\{v_1,\,\ldots,\,v_p\}$, $W=\{w_1,\,\ldots,\,w_q\}$.
By $B_{d,p,q}$ we denote the tree on vertex set $U\cup V\cup W$ such that its subtrees generated by $\{u_1\}\cup V$, $\{u_{d-1}\}\cup W$ and $U$ are isomorphic to
$K_{1,p}$, $K_{1,q}$ and $P_{d-1}$ respectively. Note that $B_{d,1,1}\simeq P_{d+1}$. It is easy to check that for all $p,q$ we have $i_M(B_{d,p,q})=i_M(P_{d+1})=\psi_{d+1}$.

\begin{theorem}
Let $d\ge3$ and let $T$ be an $(n,d)_{\txt{m.i.s}}$--minimal tree. Then $T\simeq B_{d,p,q}$ for some $p,q$.
\end{theorem}
\begin{proof}
Let $d\ge3$ and let $T$ be an $(n,d)_{\txt{m.i.s.}}$--minimal tree. Consider a diametrical path $P$ of $T$. Every m.\,i.\,s. in $P$ is contained in at least one m.\,i.\,s. in $T$,
and different m.\,i.\,s. in $P$ are contained in different m.\,i.\,s. of $T$. Thus $i_M(T)\ge i_M(P)=\psi_{d+1}$.
\par Assume there are no $p,q$ such that $T\simeq B_{d,p,q}$. Then there is a vertex $u$ in $P$, which is not adjacent to either of ends of $P$, but \emph{is} adjacent to some
$v\in V(T)\setminus V(P)$. Let $F$ be a subgraph of $T$ generated by a set $V'=\{v\}\cup(V(P)\setminus (\{u\}\cup\partial u))$.
Consider a m.\,i.\,s. $S$ in  $F$ that contains $v$. Note that we can add some vertices from $V(T)\setminus (V(P)\cup\{v\})$ to $S$ to obtain some m.\,i.\,s. in $T$,
and at the same time $S\cap V(P)$ is not a m.\,i.\,s. of $P$. But it implies $i_M(T)>i_M(P)$ which contradicts the assumption of $(n,d)_{\txt{m.i.s.}}$--minimality of $T$.
\end{proof}

\begin{stat}\label{stAllLeavesSame}
Let $T$ be an arbitrary tree. Let $T$ contain a vertex adjacent to two or more leaves, and let $u$ be one of these leaves. Then for a tree $T'$, obtained from $T$ by
removal of $u$, we have $i_M(T')=i_M(T)$.
\end{stat}
\begin{proof}
It is sufficient to note that if $u_1,\,\ldots,\,u_r$ are the leaves of $T$ having the common neighbor, then any m.\,i.\,s. of $T$ either contains all vertices $u_1,\,\ldots,\,u_r$,
or contains none of them.
\end{proof}

\begin{lemma}\label{lemOneLeafOnly}
For all $n$ and $d$ such that $4\le d<n$ every $(n,d)_{\txt{m.i.s.}}$--maximal tree contains no vertices adjacent to two or more leaves.
\end{lemma}
\begin{proof}
Let us assume that for some $d\ge4$ there is a $(n,d)_{\txt{m.i.s.}}$--maximal tree $T$, which has a vertex adjacent to two or more leaves. By removing one of these leaves,
we get a tree $T'$, that, by proposition \ref{stAllLeavesSame}, has $i_M(T')=i_M(T)$. Moreover, $\diam(T')=\diam(T)$ and $n(T')=n(T)-1$.
Observe that in every tree of diameter at least four, there exists a vertex that is not adjacent to any leaf and that is either not leaf or is a leaf not lying on the diametrical path.
Let $v$ be such a vertex in $T'$. By adding a new leaf $u$ to $T'$ and connecting it to $v$ we get the tree $T''$, that has $n(T'')=n(T)$, $\diam(T'')=\diam(T)$ and $i_M(T'')>i_M(T)$.
This contradiction completes the proof.
\end{proof}

For natural $n,\,d$ such that $4\le d\le n-1$, define a function $M(n,d)$:
$$
M(n,d)=\begin{system}{ll}\psi_{d-1}+(2^{(n-d+1)/2}-1)\psi_{d-2}, & \mbox{ for } d\ge4,\,n-d=2k+1,\,k\ge0, \medskip\\ \psi_{d-2}+\psi_d, & \mbox{ for } d\ge4,\,n-d=2, \medskip\\ 2^{(n-d)/2}\psi_{d-1}, & \mbox{ for } d\in\{5,6\}\cup[8,\infty],\,n-d=2k,\,k\ge2, \medskip\\ 2^{(n-d)/2}\psi_{d-1}+1, & \mbox{ for } d\in\{4,7\},\,n-d=2k,\,k\ge2. \end{system}
$$

\begin{stat}\label{stPreArgmax}
$\,\,$
\begin{enumerate}
\item[$1)$] For $d\ge4$ and all $n,\,n\ge d+3$, such that $2\nmid(n-d)$, we have $M(n,d)>M(n,d+1)$.
\item[$2)$] For $d\ge4$ and all $n,\,n\ge d+2$, such that $2\mid(n-d)$, we have $M(n,d)\le M(n,d+1)$, and $M(n,d)=M(n,d+1)$ only if $d=4$.
\item[$3)$] For $d\ge4$ and all $n,\,n\ge d+3$, we have $M(n,d)\ge M(n,d+2)$, and $M(n,d)=M(n,d+2)$ only if $d=5$ and $n$ is even.
\end{enumerate}
\end{stat}
\begin{proof}
$\,\,$
\begin{enumerate}
\item Let $4\le d\le n-3$ and $2\nmid(n-d)$. If $n=d+3$, then $M(n,d)-M(n,d+1)=2\psi_{d-2}-\psi_{d-1}>0$. If $n\ge d+5$ and $d\neq6$, then
$$M(n,d)-M(n,d+1)=\psi_{d-1}-\psi_{d-2}+2^{(n-d-1)/2}(2\psi_{d-2}-\psi_d)\ge\psi_{d-1}+7\psi_{d-2}-4\psi_d>0.$$
If $n\ge d+5$ and $d=6$, then $M(n,d)-M(n,d+1)=2^{(n-7)/2}>0$.
\item If $d=4$ and $n$ is even, then the equality $M(n,d)=M(n,d+1)$ can be easily checked. Let $5\le d\le n-2$ and $2\mid(n-d)$. If $n=d+2$, then $M(n,d+1)-M(n,d)=\psi_{d-1}-\psi_{d-2}>0$.
If $d\neq7$ and $n\ge d+4$, then $M(n,d+1)-M(n,d)=\psi_d-\psi_{d-1}>0$. If $d=7$ and $n\ge d+4$, then $M(n,d+1)-M(n,d)=1>0$.
\item If $4\le d\le n-3$ and $2\nmid(n-d)$, then $M(n,d)-M(n,d+2)=(2^{(n-d-1)/2}-1)(2\psi_{d-2}-\psi_d)$, which implies that $M(n,5)=M(n,7)$, and $M(n,d)>M(n,d+2)$ for $d\neq5$.
\par If $d=4$ and $n$ is even, then $M(n,d)-M(n,d+2)=1>0$. Let $5\le d\le n-3$. For $n=d+4$ we have $M(n,d)-M(n,d+2)\ge3\psi_{d-1}-2\psi_d>0$. If $n\ge d+6$ and $2\mid(n-d)$, then
$$M(n,d)-M(n,d+2)\ge2^{(n-d-2)/2}(2\psi_{d-1}-\psi_{d+1})-1>0.$$
\end{enumerate}
\end{proof}

\par Proposition \ref{stPreArgmax} implies the following.
\begin{lemma}\label{lemArgmax}
If $4\le d'<d''\le n-1$, then
\begin{enumerate}
\item[$1)$] $\Argmax_{d'\le d\le d''}M(n,d)=\begin{system}{ll}\{d'+1\}, & \mbox{ for } d'\ge5,\,2\mid(n-d'), \medskip\\ \{4,5,7\}\cap[d',d''], & \mbox{ for } d'\le5,\,2\mid n, \medskip\\\{d'\}, & \mbox{ for } d'\in\{4\}\cup[6,\infty],\,2\nmid(n-d'). \end{system}$
\item[$2)$] $\Max_{d'\le d\le d''}M(n,d)=\begin{system}{ll}M(n,d'+1), & \mbox{ for } 2\mid(n-d'), \medskip\\ M(n,d'), & \mbox{ for } 2\nmid(n-d'). \end{system}$
\item[$3)$] $\Max_{d\ge d'}M(n,d)\le M(n,4)$.
\end{enumerate}
\end{lemma}

\begin{lemma}\label{lemMaxTreesDplus2}
For $n-2=d\ge5$ every $(n,d)_{\txt{m.i.s.}}$--maximal tree $T$ is isomorphic to one of the trees on fig. \textcolor{red}{7}, and $i_M(T)=M(n,d)$.
\end{lemma}
\begin{proof}
Induction on $d$. For $5\le d\le8$ the statement of the lemma is easily checked. Let $d\ge9$ and assume the lemma to be true for the trees of diameter at most $(d-1)$.
Let $T$ be a $(d+2,d)_{\txt{m.i.s.}}$--maximal tree. Let $v$ be the only vertex of $T$ that does not lye on the diametrical path. Let $u$ be the ending vertex of diametrical path,
that is furthest to $v$. Let $u'$ be the neighbor of $u$ and let $u''$ be a vertex at distance $2$ from $u$.
Using the induction hypothesis we can estimate $i_M(T)$:
$$
i_M(T)=i_M(T\setminus\{u,u'\})+i_M(T\setminus\{u,u',u''\})\le M(n-2,d-2)+M(n-3,d-3)=M(n,d),
$$
and the equality $i_M(T)=M(n,d)$ can hold only if $T\setminus\{u,u'\}$ and $T\setminus\{u,u',u''\}$ are isomorphic to some trees on fig. \textcolor{red}{7}.
But this can only be if $T$ is also isomorphic to a tree on fig. \textcolor{red}{7c}.
\end{proof}

\begin{lemma}\label{lemMaxTreesD5}
Every $(n,5)_{\txt{m.i.s.}}$--maximal tree is isomorphic to one of the trees on fig. \textcolor{red}{3a,3b}.
\end{lemma}
\begin{proof}
Let $T$ be an arbitrary $(n,5)_{\txt{m.i.s.}}$--maximal tree. If $T$ is as shown on fig. \textcolor{red}{3a,3b}, then it is easy to check that $i_M(T)=M(n,5)$.
Assume that $T$ is not isomorphic to any of the trees on fig. \textcolor{red}{3a,3b}. Then lemma \ref{lemOneLeafOnly} implies that it it is sufficient to consider the following two cases:
\begin{enumerate}
\item $T$ is as shown on fig. \textcolor{red}{3c}, where $n\ge9$ and $p\ge2$. In this case $$i_M(T)=2^{(n-5)/2}(2+2^{1-p})<3\cdot2^{(n-5)/2}=M(n,5).$$
\item $T$ is as shown on fig. \textcolor{red}{3d}, where $n\ge8$. In this case $$i_M(T)=2^{(n-4)/2}+2^p+2^{(n-2p-4)/2}\le2+3\cdot2^{(n-6)/2}<1+4\cdot2^{(n-6)/2}=M(n,5).$$
\end{enumerate}
Thus, in each of two cases we get a contradiction with maximality of $T$, which completes the proof.
\end{proof}

Trivially, every tree of diameter $1$ or $2$ contains exactly $2$ m.\,i.\,s., and every tree of diameter $3$ contains $3$ m.\,i.\,s.
For $d\ge4$ the complete characterization of $(n,d)_{\txt{m.i.s.}}$--maximal trees is provided by the following theorem.

\begin{theorem}
For all $4\le d\le n-1$ any $(n,d)_{\txt{m.i.s.}}$--maximal tree $T$ has $i_M(T)=M(n,t)$ and is isomorphic to one of the trees listed in the following table:
{\textup{
\begin{center}
\begin{tabular}{|c|l|l|}\hline
$d$&$(n-d)$&extremal trees\\ \hline
$4$&$2k+1$ ($k\ge1$)&fig. \textcolor{red}{2a} \\ \hline
$4$&$2k$ ($k\ge1$)&fig. \textcolor{red}{2b} \\ \hline
$5$&$2k+1$ ($k\ge1$)&fig. \textcolor{red}{3b} ($1\le p\le\frac{n-4}{2}$) \\ \hline
$5$&$2k$ ($k\ge1$)&fig. \textcolor{red}{3a} \\ \hline
$6$&$2k+1$ ($k\ge1$)&fig. \textcolor{red}{1a} \\ \hline
$6$&$2$&fig. \textcolor{red}{7a}, \textcolor{red}{7c} \\ \hline
$6$&$2k$ ($k\ge2$)&fig. \textcolor{red}{4d} ($1\le p\le\frac{n-6}{2}$) \\ \hline
$7$&$2k+1$ ($k\ge1$)&fig. \textcolor{red}{5b} ($1\le p\le\frac{n-6}{2}$) \\ \hline
$7$&$2k$ ($k\ge1$)&fig. \textcolor{red}{5a} \\ \hline
$8$&$2k+1$ ($k\ge1$)&fig. \textcolor{red}{1a} \\ \hline
$8$&$2$&fig. \textcolor{red}{7b}, \textcolor{red}{7c} \\ \hline
$8$&$2k$ ($k\ge2$)&fig. \textcolor{red}{1b} \\ \hline
$\ge9$&$2k+1$ ($k\ge1$)&fig. \textcolor{red}{1a} \\ \hline
$\ge9$&$2$&fig. \textcolor{red}{7c} \\ \hline
$\ge9$&$2k$ ($k\ge2$)&fig. \textcolor{red}{1b} \\ \hline
\end{tabular}
\end{center}
}}
\end{theorem}

\begin{proof}
The statement of the theorem for $d=4$ follows directly from lemma \ref{lemOneLeafOnly}, and for $d=5$ is corollary of lemma \ref{lemMaxTreesD5}. Moreover, the statement is trivial for
$n=d+1$ and follows from lemma \ref{lemMaxTreesDplus2} for $n=d+2$.
\par Let $d\ge 6$, $n\ge d+3$, and assume that the theorem holds for all pairs $(n'',d)$ such that $n''<n$, and all pairs $(n',d')$ such that $d'\le d-1$.
Let us prove that then the theorem also holds for $(n,d)$. Let $T$ be an arbitrary $(n,d)_{\txt{m.i.s.}}$--maximal tree, and let $P$ be some diametrical path of $T$.
First we will show that for every vertex $v$, which is not contained in $P$, the minimal distance from $v$ to vertices of $P$ does not exceed $2$. Let us assume the contrary
and show that under this assumption the inequality $M(n,d)-i_M(T)>0$ holds, contradicting the choice of $T$. Lemma \ref{lemOneLeafOnly} implies that only these cases are possible:
\begin{enumerate}
\item $T$ is as shown on fig. \textcolor{red}{9a}, where $\diam(T')=d$ and $1\le t\le \frac{n-d-2}{2}$. In this case $M(n,d)-i_M(T)\ge D_1(n,d)$, where
$$D_1(n,d)=M(n,d)-M(n-2,d)-M(n-2t-1,d)\cdot2^{t-1}.$$
Consider the following subcases:
    \begin{enumerate}
    \item $n=d+4$. Then $t=1$ and $D_1(n,d)\ge4\psi_{d-1}-\psi_{d-2}-\psi_d-\psi_{d+1}>0$.
    \item $n=d+2k$ with $k\ge3$. Then
    $$
    \begin{array}{rl}
    D_1(n,d)&=2^{(n-d)/2}\psi_{d-1}-2^{(n-d-2)/2}(\psi_{d-1}+\psi_{d-2})-2^{t-1}(\psi_{d-1}-\psi_{d-2})\ge\\
    &\ge 2^{(n-d)/2}\psi_{d-1}-2^{(n-d-2)/2}(\psi_{d-1}+\psi_{d-2})-2^{(n-d-4)/2}(\psi_{d-1}-\psi_{d-2})=\\
    &=2^{(n-d-4)/2}(\psi_{d-1}-\psi_{d-2})>0.
    \end{array}
    $$
    \item $n=d+2k+1$ with $k\ge2$, and $t=\frac{n-d-3}{2}$. Then $D_1(n,d)=2^{(n-d-5)/2}(3\psi_{d-2}-\psi_d)>0$.
    \item $n=d+2k+1$ with $k\ge3$, and $t\le\frac{n-d-5}{2}$. Then $D_1(n,d)\ge2^{(n-d-7)/2}(8\psi_{d-2}-4\psi_{d-1}-1)>0$.
    \end{enumerate}
\item $T$ is as shown on fig. \textcolor{red}{9b}, where $\diam(T')=d$ and $1\le t\le \frac{n-d-3}{2}$. In this case we have $M(n,d)-i_M(T)\ge D_2(n,d)$, where
$$D_2(n,d)=M(n,d)-M(n-2,d)-M(n-2t-2,d)\cdot2^{t-1}.$$
Consider the following subcases:
    \begin{enumerate}
    \item $n=d+2k+1$ with $k\ge2$. Then
    $$
    \begin{array}{rl}
    D_2(n,d)&=2^{(n-d-3)/2}\psi_{d-2}-2^{t-1}(\psi_{d-1}-\psi_{d-2})\ge\\
    &\ge 2^{(n-d-3)/2}\psi_{d-2}-2^{(n-d-5)/2}(\psi_{d-1}-\psi_{d-2})=\\
    &=2^{(n-d-5)/2}(3\psi_{d-2}-\psi_{d-1})>0.
    \end{array}
    $$
    \item $n=d+2k$ with $k\ge3$, and $t=\frac{n-d-4}{2}$. Then $D_2(n,d)=2^{(n-d-6)/2}(4\psi_{d-1}-\psi_d-\psi_{d-2})>0$.
    \item $n=d+2k$ with $k\ge3$, and $t\le\frac{n-d-6}{2}$. Then $D_2(n,d)>2^{(n-d-4)/2}(\psi_{d-1}-1)>0$.
    \end{enumerate}
\end{enumerate}
\par In any case, the assumption of existence of a vertex distanced from $P$ by $2$ or more, contradicts the $(n,d)_{\txt{m.i.s.}}$--maximality of $T$. From this and from
lemma \ref{lemOneLeafOnly} it follows that in the remaining part of the theorem we may assume that every vertex in $T$ is at most $2$ edges away from the diametrical path, and
every vertex of $T$ is neighboring at most one leaf.
\par We shall consider the cases, when $T$ has diameter $6$ or $7$, separately from the general case.
\begin{enumerate}
\item $\diam(T)=6$. Assume that $T$ is not isomorphic to any of the trees on fig. \textcolor{red}{1a}, \textcolor{red}{1b} or \textcolor{red}{4d}. Then the following cases are possible:
    \begin{enumerate}
    \item $T$ is as shown on fig. \textcolor{red}{9b}, where $1\le t\le \frac{n-6}{2}$ and $\diam(T')\ge3$. If $T'$ has only four vertices, then $i_M(T)=2+3\cdot2^{(n-6)/2}<M(n,6)$.
    Let $T'$ have at least $5$ vertices. Then lemma \ref{lemArgmax} and induction hypothesis imply
    $$M(n,6)-i_M(T)\ge M(n,6)-\max_{d\in\{5,6\}} M(n-2,d)-2^{t-1}M(n-2t-2,4).$$
    If $n$ is even, then $\max_{d\in\{5,6\}} M(n-2,d)=M(n-2,5)$ and
    $$
    M(n,6)-i_M(T)\ge M(n,6)-M(n-2,5)-2^{t-1}M(n-2t-2,4)\ge2^{(n-8)/2}-1>0.
    $$
    If $n$ is odd, then $\max_{d\in\{5,6\}} M(n-2,d)=M(n-2,6)$ and
    $$
    M(n,6)-i_M(T)\ge M(n,6)-M(n-2,6)-2^{t-1}M(n-2t-2,4)=2^{(n-7)/2}>0.
    $$
    \item $T$ is either isomorphic to one of the trees on fig. \textcolor{red}{4a}, \textcolor{red}{4b}, or isomorphic to the tree on fig. \textcolor{red}{4c} with $2\le p\le\frac{n-7}{2}$.
    The following table shows that in all these cases we have $i_M(T)<M(n,6)$:
    \begin{center}
    \begin{tabular}{|c|c|c|}\hline
    fig. & $i_M(T)$ & lower bound for $(M(n,6)-i_M(T))$ \\ \hline
    \textcolor{red}{4a}&$2^{(n-3)/2}+2^p+2^{(n-3-2p-2q)/2}-1$&$2^{(n-7)/2}$\\ \hline
    \textcolor{red}{4b}&$2^{(n-4)/2}+2^{(n-4-2q)/2}$&$2^{(n-6)/2}$\\ \hline
    \textcolor{red}{4c}&$2^{(n-3)/2}+2^p+2^{(n-3-2p)/2}-1$&$2^{(n-7)/2}-2$\\ \hline
    \end{tabular}
    \end{center}
    \end{enumerate}
\item $\diam(T)=7$. The following subcases are possible:
    \begin{enumerate}
    \item\label{itm7_Pre2} $T$ is as shown on fig. \textcolor{red}{9b}, where $1\le t\le \frac{n-7}{2}$ and $\diam(T')\ge4$. Then lemma \ref{lemArgmax} and induction hypothesis imply
    $$M(n,7)-i_M(T)\ge M(n,7)-\max_{d\in\{6,7\}} M(n-2,d)-2^{t-1}M(n-2t-2,4).$$
    If $n$ is even, then $\max_{d\in\{6,7\}} M(n-2,d)=M(n-2,7)$ and
    $$
    M(n,7)-i_M(T)\ge M(n,7)-M(n-2,7)-2^{t-1}M(n-2t-2,4)\ge3\cdot2^{(n-10)/2}>0.
    $$
    If $n$ is odd, then $\max_{d\in\{6,7\}} M(n-2,d)=M(n-2,6)$ and
    $$
    M(n,7)-i_M(T)\ge M(n,7)-M(n-2,6)-2^{t-1}M(n-2t-2,4)=0,
    $$
    and lemma \ref{lemArgmax} and induction hypothesis imply that the equality $$M(n,7)-i_M(T)=M(n,7)-M(n-2,6)-2^{t-1}M(n-2t-2,4)$$ can only hold if $t=1$ and $T'$
    is as shown on fig. \textcolor{red}{2a}, which is possible only if $T$ is as shown on fig. \textcolor{red}{5a}.
    \item $T$ is as shown on fig. \textcolor{red}{9a}, where $1\le t\le \frac{n-6}{2}$. If $n$ is even, then lemma \ref{lemArgmax} and induction hypothesis imply
    $$
    i_M(T)\le M(n-2,5)+2^{t-1}M(n-2t-1,4)=M(n,7),
    $$
    and the equality $i_M(T)=M(n,7)$ can occur only if $T'$ is as shown on fig. \textcolor{red}{2a}, which can only happen if $T$ is isomorphic to the tree on fig. \textcolor{red}{5b}.
    \par Suppose now that $n$ is odd, and $T$ is not isomorphic to the tree on fig. \textcolor{red}{5a}, and cannot be considered in the scope of the case \ref{itm7_Pre2}.
    Then $T$ is as shown on fig. \textcolor{red}{5c} for some $q,r\ge0$, and we have
    $$
    i_M(T)=2^{(n-5)/2}+2^{p+q}+2^{(n-2q-5)/2}\le 5\cdot2^{(n-7)/2}<M(n,7).
    $$
    \end{enumerate}
\end{enumerate}
\par All cases when $d=\diam(T)\le7$ were considered above, and in the remaining part of the proof we will assume $d\ge8$.
Fix some diametrical path $P$ in $T$. Let $w,\,w',\,u,\,u',\,u''$ be the successive vertices of $P$, where $w$ is an end of $P$.
We will assume that if $\tilde{w},\,\tilde{w}',\,\tilde{u},\,\tilde{u}',\,\tilde{u}''$ are successive vertices of $P$ such that $\tilde{w}$ is the ending vetrex of $P$ opposite to $w$,
then the tuple of degrees $(\deg w,\,\deg w',\,\deg u,\,\deg u',\,\deg u'')$ is lexicographically no less than the
tuple $(\deg \tilde{w},\,\deg \tilde{w}',\,\deg \tilde{u},\,\deg \tilde{u}',\,\deg \tilde{u}'')$. We will split the proof into the consideration of the following cases:
\begin{enumerate}
\item\label{itmUt2L0} Vertex $u$ is adjacent to $t$ paths on two vertices, where $t\ge 2$, and $u$ is adjacent to no leaves. In this case $T$ is as shown on fig. \textcolor{red}{9a},
where $\diam(T')\ge d-3$. We have
\begin{equation}\label{eqRecurT1}
i_M(T)\le M(n-2,d)+2^{t-1}\cdot\max_{d'\ge d-3}M(n-2t-1,d').
\end{equation}
    \begin{enumerate}
    \item If $2\nmid(n-d)$ then lemma \ref{lemArgmax} implies $\max_{d'\ge d-3}M(n-2t-1,d')\le M(n-2t-1,d-3)$. Then \eqref{eqRecurT1} implies
    $$
    \begin{array}{rl}
    M(n,d)-i_M(T)&\ge M(n,d)-M(n-2,d)-2^{t-1}\cdot M(n-2t-1,d-3)=\\
    &=2^{(n-d-1)/2}\psi_{d-2}-2^{(n-d+1)}\psi_{d-5}-2^{t-1}(\psi_{d-4}-\psi_{d-5})\ge\\
    &\ge 2^{(n-d-1)/2}\psi_{d-2}-2^{(n-d+1)/2}\psi_{d-5}-2^{(n-d-1)/2}(\psi_{d-4}-\psi_{d-5})=0,
    \end{array}
    $$
    and equality $i_M(T)=M(n,d)$ can only hold if $t=\frac{n-d+1}{2}$, that is if $T$ is as shown of fig. \textcolor{red}{1a}.
    \item If $2\mid(n-d)$ then lemma \ref{lemArgmax} implies $\max_{d'\ge d-3}M(n-2t-1,d')\le M(n-2t-1,d-2)$. The following two subcases are possible.
        \begin{enumerate}
        \item $n\ge d+6$. Then
        $$
        \begin{array}{rl}
        M(n,d)-i_M(T)&\ge M(n,d)-M(n-2,d)-2^{t-1}\cdot M(n-2t-1,d-2)=\\
        &=(2^{(n-d-2)/2}-2^{t-1})(\psi_{d-3}-\psi_{d-4})\ge 0,
        \end{array}
        $$
        and the equality $i_M(T)=M(n,d)$ can hold only if $t=\frac{n-d}{2}$ and $\diam(T')=d-2$. But then $T$ is as shown on fig. \textcolor{red}{1b}.
        \item $n=d+4$. Then $T$ is isomorphic to one of the trees on fig. \textcolor{red}{8a-8c}. If $T$ is as shown on fig. \textcolor{red}{8a}, then
        $$M(n,d)-i_M(T)=4\psi_{d-1}-2(\psi_d+\psi_{d-5})-\psi_{d-3}>0.$$
        If $T$ is as shown on fig. \textcolor{red}{8b}, then $M(n,d)-i_M(T)=2\psi_{d-1}-4\psi_{d-4}>0$. If $T$ is as shown on fig. \textcolor{red}{8c}, then for $8\le d\le 10$
        the statement of the theorem is easily checked, and for $d\ge11$ the induction hypothesis implies
        $$M(n,d)-i_M(T)\ge M(n,d)-M(n-2,d-2)-M(n-3,d-3)=0,$$
        with $i_M(T)=M(n,d)$ iff $T$ is as shown on fig. \textcolor{red}{1b}.
        \end{enumerate}
    \end{enumerate}
\item Vertex $u$ is adjacent to $t,\,t\ge2$, paths on two vertices and is adjacent to exactly one leaf. In this case $T$ is as shown on fig. \textcolor{red}{9b}, where $\diam(T')\ge d-3$.
We have
\begin{equation}\label{eqRecurT2}
i_M(T)\le M(n-2,d)+2^{t-1}\cdot\max_{d'\ge d-3}M(n-2t-2,d').
\end{equation}
    \begin{enumerate}
    \item If $2\nmid(n-d)$, then lemma \ref{lemArgmax} implies $\max_{d'\ge d-3}M(n-2t-1,d')\le M(n-2t-2,d-2)$. Then from \eqref{eqRecurT2} it follows that
    $$
    \begin{array}{rl}
    M(n,d)-i_M(T)&\ge M(n,d)-M(n-2,d)-2^{t-1}\cdot M(n-2t-2,d-2)=\\
    &=2^{(n-d-1)/2}(\psi_{d-2}-\psi_{d-4})-2^{t-1}(\psi_{d-3}-\psi_{d-4})\ge\\
    &\ge 2^{(n-d-1)/2}(\psi_{d-2}-\psi_{d-4})-2^{(n-d-3)/2}(\psi_{d-3}-\psi_{d-4})=\\
    &=2^{(n-d-3)/2}(2\psi_{d-2}-\psi_{d-1})>0.
    \end{array}
    $$
    \item If $2\mid(n-d)$, then lemma \ref{lemArgmax} implies $\max_{d'\ge d-3}M(n-2t-2,d')\le M(n-2t-2,d-3)$. Then \eqref{eqRecurT2} implies
    $$
    \begin{array}{rl}
    M(n,d)-i_M(T)&\ge M(n,d)-M(n-2,d)-2^{t-1}\cdot M(n-2t-2,d-3)=\\
    &=2^{(n-d-2)/2}\psi_{d-1}-2^{t-1}(\psi_{d-4}-\psi_{d-5})-2^{(n-d)/2}\psi_{d-5}\ge\\
    &\ge2^{(n-d-2)/2}\psi_{d-1}-2^{(n-d-2)/2}(\psi_{d-4}-\psi_{d-5})-2^{(n-d)/2}\psi_{d-5}=0,\\
    \end{array}
    $$
    and for equality $i_M(T)=M(n,d)$ to hold, it is necessary that $t=\frac{n-d}{2}$. But for $t=\frac{n-d}{2}$ we have
    $$M(n,d)-i_M(T)=2^{(n-d)/2}(\psi_{d-1}-\psi_{d-2})-\psi_{d-3}\ge 4(\psi_{d-1}-\psi_{d-2})-\psi_{d-3}>0.$$
    \end{enumerate}
\item\label{itmUt1L1} $u$ is adjacent to one path on two vertices and one leaf. In this case $T$ is as shown on fig. \textcolor{red}{9c}, where $\diam(T')\ge d-3$. Then,
like in the previous case, we apply the induction hypothesis and lemma \ref{lemArgmax}:
    \begin{enumerate}
    \item If $2\mid(n-d)$, then
    $$
    \begin{array}{rl}
    M(n,d)-i_M(T)&\ge M(n,d)-M(n-2,d-1)-M(n-4,d-3)=\\
    &=(2^{(n-d)/2}-1)(\psi_{d-4}-\psi_{d-5})-\psi_{d-2}+\psi_{d-3}\ge \\
    &\ge 3(\psi_{d-4}-\psi_{d-5})-\psi_{d-2}+\psi_{d-3}>0.
    \end{array}
    $$
    \item If $2\nmid(n-d)$, then
    $$
    \begin{array}{rl}
    M(n,d)-i_M(T)&\ge M(n,d)-M(n-2,d-1)-M(n-4,d-2)\ge\\
    &\ge (2^{(n-d-1)/2}-1)(\psi_{d-2}-\psi_{d-4})+\psi_{d-1}-\psi_{d-3}-1>0.
    \end{array}
    $$
    \end{enumerate}
\item It suffices to consider the case when $u$ is adjacent to one path on two vertices and has no neighboring leaves (that is $\deg u=2$). Consider $u'$, the neighbor of $u$ which
is at distance $3$ from the end of diametrical path $P$. The following four subcases are possible:
    \begin{enumerate}
    \item $u'$ is adjacent to some path on two vertices. Then $M(n,d)-i_M(T)\ge D_3(n,d)$, where $$D_3(n,d)=M(n,d)-M(n-2,d-1)-M(n-3,d-1).$$
        \begin{enumerate}
        \item If $n=d+3$, then $D_3(n,d)=3\psi_{d-2}-\psi_{d-3}-\psi_d>0.$
        \item If $n=d+4$, then for $d\ge12$ we have $$D_3(n,d)=2\psi_{d-4}-\psi_{d-3}-\psi_{d-5}>0,$$
        and for $8\le d\le11$ the inequality $i_M(T)<M(n,d)$ can be easily checked by hand.
        \item If $2\nmid(n-d)$ and $n\ge d+5$, then
        $$D_3(n,d)\ge(2^{(n-d-1)/2}-1)(\psi_{d-2}-\psi_{d-3})+\psi_{d-1}-\psi_{d-2}-1>0.$$
        \item If $2\mid(n-d)$ and $n\ge d+6$, then
        $$
        \begin{array}{rl}
        D_3(n,d)&=2^{(n-d-2)/2}(2\psi_{d-4}-\psi_{d-2})-\psi_{d-2}+\psi_{d-3}-1\ge\\
        &\ge4(2\psi_{d-4}-\psi_{d-2})-\psi_{d-2}+\psi_{d-3}-1>0.
        \end{array}
        $$
        \end{enumerate}
    \item\label{itmU2UL1} $u'$ is adjacent to one leaf and no paths on two vertices. In this case proposition \ref{stAllLeavesSame} and the induction hypothesis imply $i_M(T)\le 2M(n-3,d-2)<M(n,d)$.
    \item The degree of $u'$ is $2$. Then we consider the vertex $u''$ neighboring $u'$, which is at distance $4$ of the end of $P$. Firstly we consider the case,
    when $u''$ is adjacent to no paths on two vertices. If $d\notin\{9,10\}$, or $2\nmid(n-d)$, then the theorem follows from induction hypothesis and equality $M(n-2,d-2)+M(n-3,d-3)=M(n,d)$.
    The cases $d=9,\,2\nmid n$ and $d=10,\,2\mid n$ have to be considered separately due to ``non-standard'' behavior of $M(n,d)$ for $d=7,\,2\nmid n$:
        \begin{enumerate}
        \item $d=9$ and $2\nmid n$. If $T\setminus\{w,w'\}$ is not isomorphic to a tree on fig. \textcolor{red}{5a} and at the same time $T\setminus\{w,w',u\}$ is not isomorphic
        to the tree on fig. \textcolor{red}{4d}, then induction hypothesis implies $i_M(T)\le M(n-2,7)+M(n-3,6)-2<M(n,7)$.
        If $T\setminus\{w,w'\}$ was as shown on fig. \textcolor{red}{5a}, or $T\setminus\{w,w',u\}$ was as shown on fig. \textcolor{red}{4d},
        then $T$ would have been in the scope of the previously considered cases \ref{itmUt2L0}, \ref{itmUt1L1}, \ref{itmU2UL1}.
        \item $d=10$ and $2\mid n$. If $T\setminus\{w,w'\}$ is not isomorphic to a tree on fig. \textcolor{red}{1b} and at the same time $T\setminus\{w,w',u\}$ is not isomorphic
        to the tree on fig. \textcolor{red}{5a}, then induction hypothesis implies $i_M(T)\le M(n-2,8)+M(n-3,7)-2<M(n,7)$. If $T\setminus\{w,w'\}$ is isomorphic to the tree
        on fig. \textcolor{red}{1b}, then $T$ is isomorphic to the tree on fig. \textcolor{red}{5a}. If $T\setminus\{w,w',u\}$ was as shown on fig. \textcolor{red}{5a},
        then $T$ would have been in the scope of the previously considered cases \ref{itmUt2L0}, \ref{itmUt1L1}, \ref{itmU2UL1}.
        \end{enumerate}
    \item The only case which suffices to be considered is when $\deg u'=2$, and $u''$ is adjacent to at least one path on two vertices. The following four subcases are possible:
        \begin{enumerate}
        \item $\diam(T)=8$. Then $n\ge11$ and $T$ is isomorphic to one of the trees on fig. \textcolor{red}{6a} or \textcolor{red}{6b}.
        The following table shows that in these both cases we would have $i_M(T)<M(n,8)$.
        \begin{center}
        \begin{tabular}{|c|c|c|}\hline
        fig. & $i_M(T)$ & $M(n,8)-i_M(T)$ \\ \hline
        \textcolor{red}{6a}&$9\cdot2^{(n-9)/2}+3$&$2^{(n-9)/2}-1$\\ \hline
        \textcolor{red}{6b}&$9\cdot2^{(n-10)/2}+4$&$5\cdot2^{(n-10)/2}-4$\\ \hline
        \end{tabular}
        \end{center}
        \item $\diam(T)=9$. Then $n\ge14$ and $T$ is isomorphic to one of the trees on fig. \textcolor{red}{6c--6e}.
        The following table shows that in all these cases we would have $i_M(T)<M(n,9)$.
        \begin{center}
        \begin{tabular}{|c|c|c|}\hline
        fig. & $i_M(T)$ & lower bound for $(M(n,9)-i_M(T))$ \\ \hline
        \textcolor{red}{6c}&$9\cdot2^{(n-10)/2}+3(2^p+2^{(n-10-2p)/2})+1$&$7\cdot2^{(n-12)/2}-5$\\ \hline
        \textcolor{red}{6d}&$9\cdot2^{(n-11)/2}+3\cdot2^p+6\cdot2^{(n-11-2p)/2}$&$3\cdot2^{(n-9)/2}-6$\\ \hline
        \textcolor{red}{6e}&$9\cdot2^{(n-12)/2}+6(2^p+2^{(n-12-2p)/2})$&$2^{(n-4)/2}-10$\\ \hline
        \end{tabular}
        \end{center}
        \item $d\ge10$ and $T$ is as shown on fig. \textcolor{red}{9d}. If $2\mid(n-d)$ and $n\ge d+6$, then
        $$
        \begin{array}{rl}
        M(n,d)-i_M(T)&\ge M(n,d)-M(n-2,d)-3\cdot2^{t-1}\cdot M(n-2t-5,d-4)=\\
        &=2^{(n-d-2)/2}(\psi_{d-1}-3\psi_{d-6})-3\cdot2^{t-1}(\psi_{d-5}-\psi_{d-6})\ge\\
        &\ge2^{(n-d-2)/2}(\psi_{d-1}-3\psi_{d-6})-3\cdot2^{(n-d-4)/2}(\psi_{d-5}-\psi_{d-6})=\\
        &=2^{(n-d-4)/2}(2\psi_{d-4}-\psi_{d-3})>0.
        \end{array}
        $$
        If $n=d+4$, then $t=1$ and $$M(n,d)-i_M(T)\ge M(d+4,d)-M(d+2,d)-3M(d-3,d-4)=4\psi_{d-1}-2\psi_d-2\psi_{d-3}>0.$$
        If $2\nmid(n-d)$, then
        $$
        \begin{array}{rl}
        M(n,d)-i_M(T)&\ge M(n,d)-M(n-2,d)-3\cdot2^{t-1}\cdot M(n-2t-5,d-5)=\\
        &=2^{(n-d-1)/2}(\psi_{d-2}-3\psi_{d-7})-3\cdot2^{t-1}(\psi_{d-6}-\psi_{d-7})\ge\\
        &\ge2^{(n-d-1)/2}(\psi_{d-2}-3\psi_{d-7})-3\cdot2^{(n-d-3)/2}(\psi_{d-6}-\psi_{d-7})=\\
        &=2^{(n-d-3)/2}(2\psi_{d-2}-3\psi_{d-4})>0.
        \end{array}
        $$
        \item $d\ge10$ and $T$ is as shown on fig. \textcolor{red}{9e}. Then if $2\mid(n-d)$ and $n\ge d+6$, we have
        $$
        \begin{array}{rl}
        M(n,d)-i_M(T)&\ge M(n,d)-M(n-2,d)-3\cdot2^{t-1}\cdot M(n-2t-6,d-5)=\\
        &=2^{(n-d-2)/2}(\psi_{d-1}-3\psi_{d-7})-3\cdot2^{t-1}(\psi_{d-6}-\psi_{d-7})\ge\\
        &\ge2^{(n-d-2)/2}(\psi_{d-1}-3\psi_{d-7})-3\cdot2^{(n-d-4)/2}(\psi_{d-6}-\psi_{d-7})=\\
        &=2^{(n-d-4)/2}(2\psi_{d-3}-\psi_{d-4})>0.
        \end{array}
        $$
        If $n=d+4$, then $t=1$ and $$M(n,d)-i_M(T)\ge M(d+4,d)-M(d+2,d)-3M(d-4,d-5)=\psi_{d-1}+2\psi_{d-3}-2\psi_{d-2}>0.$$
        If $2\nmid(n-d)$, then
        $$
        \begin{array}{rl}
        M(n,d)-i_M(T)&\ge M(n,d)-M(n-2,d)-3\cdot2^{t-1}\cdot M(n-2t-6,d-4)=\\
        &=2^{(n-d-3)/2}(2\psi_{d-2}-3\psi_{d-6})-3\cdot2^{t-1}(\psi_{d-5}-\psi_{d-6})\ge\\
        &\ge2^{(n-d-3)/2}(2\psi_{d-2}-3\psi_{d-6})-3\cdot2^{(n-d-5)/2}(\psi_{d-5}-\psi_{d-6})=\\
        &=2^{(n-d-5)/2}(4\psi_{d-2}-3\psi_{d-3})>0.
        \end{array}
        $$
        \end{enumerate}
    \end{enumerate}
\end{enumerate}
\end{proof}

\section*{List of figures}
\begin{tabular}{cc}
\includegraphics{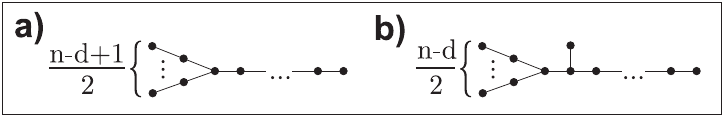}&\includegraphics{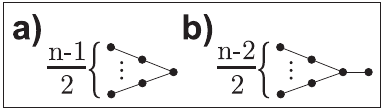}\\
Fig. 1&Fig. 2\\
\vphantom{A}& \\
\includegraphics{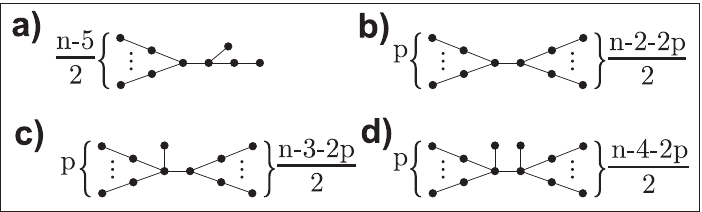}&\includegraphics{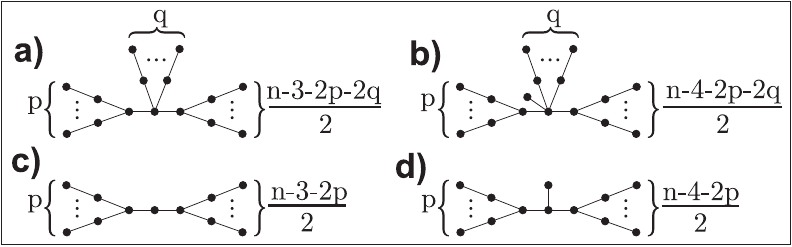}\\
Fig. 3&Fig. 4\\
\vphantom{A}& \\
\includegraphics{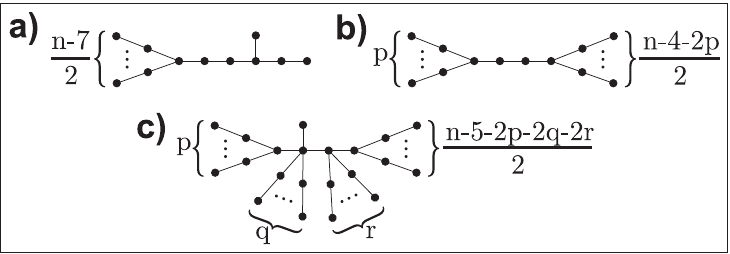}&\includegraphics{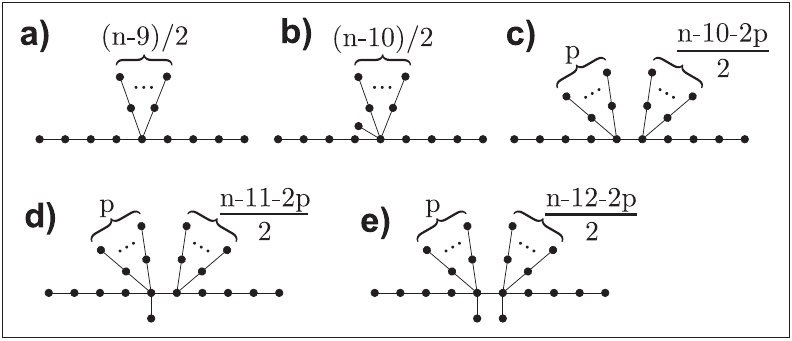}\\
Fig. 5&Fig. 6\\
\vphantom{A}& \\
\includegraphics{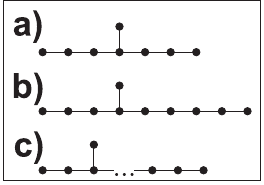}&\includegraphics{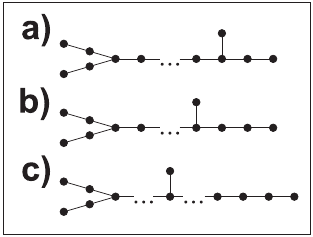}\\
Fig. 7&Fig. 8\\
\vphantom{A}& \\
\includegraphics{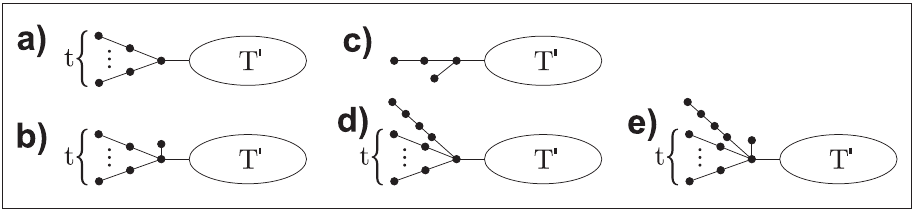}& \\
Fig. 9&
\end{tabular}
\end{document}